\documentclass[12pt]{article}
\usepackage{latexsym,amssymb,amsmath,amsfonts,amsthm}
\usepackage[notref,notcite]{showkeys}
\newtheorem{theorem}{Theorem}

\newtheorem{lemma}{Lemma}[section]

\newtheorem{proposition}{Proposition}
\newtheorem{cor}{Corollary}

\def\beq{ \begin{equation} }
\def\eeq{ \end{equation} }
\def\beqa{\begin{eqnarray}}
\def\eeqa{\end{eqnarray}}
\def\beqax{\begin{eqnarray*}}
\def\eeqax{\end{eqnarray*}}
\def\mn{\medskip\noindent}
\def\ms{\medskip}

\def\ep{\epsilon}

\def\square{\vcenter{\vbox{\hrule height .4pt
  \hbox{\vrule width .4pt height 5pt \kern 5pt
        \vrule width .4pt} \hrule height .4pt}}}

\def\ZZ{\mathbb{Z}}

\def\bb{\mathbb}

\begin{document}

\title{A first order phase transition in \\
the threshold $\theta\ge 2$ contact process on \\
random $r$-regular graphs and $r$-trees}
 \author{Shirshendu Chatterjee\\
Courant Institute of Mathematical Sciences, New York University \\
251 Mercer Street, New York, N.Y. 10012-1185  
\thanks{This work was part of his Ph.D.~thesis written at Cornell University.}\\
and\\
Rick Durrett \\
James B.~Duke Professor of Mathematics \\
Box 90320, Duke University, Durham, NC 27708-0320
\thanks{Both authors were partially supported by grants DMS 0704996 and DMS 1005470 from the
probability program at NSF.} }

\maketitle

\begin{abstract}
We consider the discrete time threshold $\theta \ge 2$ contact process on a
random $r$-regular graph on $n$ vertices. In this process, a vertex
with at least $\theta$ occupied neighbors at time $t$ will be occupied at
time $t+1$ with probability $p$, and vacant otherwise. We show that
if $\theta \ge 2$ and $r \ge \theta+2$, $\ep_1$ is small and $p \ge p_1(\ep_1)$, then
starting from all vertices occupied the fraction of occupied
vertices is $\ge 1-2\ep_1$ up to time $\exp(\gamma_1(r)n)$ with
probability $\ge 1 - \exp(-\gamma_1(r)n)$. In the other direction,
we show that for $p_2<1$ there is an $\ep_2(p_2)>0$ so that if $p
\le p_2$ and the number of occupied vertices in the initial
configuration is $\le \ep_2(p_2)n$, then with high probability all
vertices are vacant at time $C_2(p_2) \log n$. These two conclusions
imply that on the random $r$-regular graph there cannot be a
quasi-stationary distribution with density of occupied vertices in
$(0,\ep_2(p_1))$, and allow us to conclude that the process on the
$r$-tree has a first order phase transition.
\end{abstract}

 \vfill
 \clearpage

\section{Introduction}

The linear contact process was introduced by Harris in (1974) and has been studied extensively since then, see part I of Liggett (1999). In that model the state of the system is $\xi_t : \ZZ^d \to \{0,1\}$ where 1 = occupied and 0 = vacant. Occupied sites become vacant at rate 1, while a vacant site becomes occupied at rate $\lambda k$ if it has $k$ occupied neighbors.

In this paper, we will be concerned with particle systems that are versions of the contact process with sexual reproduction. In our first two models, occupied sites become vacant at rate 1. Perhaps the most natural generalization is the quadratic contact process in which a vacant site with $k$ occupied neighbors becomes occupied at rate $\lambda \binom{k}{2}$. However, we will primarily be concerned with the the  threshold-$\theta$ contact process in which a vacant site become occupied at rate $\lambda$ if it has $k \ge \theta$ occupied neighbors. The threshold-1 contact process has been studied and found to have the same qualitative behavior as the linear contact process, so we expect that the threshold-2 and quadratic contact process will as well.  

Being attractive processes, each of our models with sexual reproduction have a translation invariant upper invariant measure, $\xi^1_\infty$, that is the limit of the system starting from all 1's. See Liggett (1985, 1999) for more details about this and the results we cite in the questions below. There are three basic questions for our model.

\mn
Q1. Let $\xi^p_t$ be system starting from product measure with density $p$, i.e., $\xi^p_0(x)$ are independent and $= 1$ with probability $p$. Does $\xi^p_t$ die out for small $p$? That is, do we have $P(\xi_t^p(x)=1) \to 0$ if $p \le p_0(\lambda)$?

\mn
Q2. Let $\rho(\lambda) = P( \xi^1_\infty(x) = 1)$ and let $\lambda_c = \inf\{ \lambda : \rho(\lambda) > 0 \}$. Is $\rho(\lambda)$ discontinuous at $\lambda_c$?  If so, then soft results imply that $P( \xi^1_\infty(x)=1 ) > 0$ when $\lambda=\lambda_c$.

\mn
Q3. Let $\xi^{0,\beta}_\infty$ be the limit as $t\to\infty$ for the system starting from all 0's when there are spontaneous births at rate $\beta$. Is $\lim_{\beta\to 0} P( \xi^{0,\beta}_\infty(x) = 1 ) = 0$? If so, we say that 0 is stable under perturbation, and it follows that there are two nontrivial stationary distributions when $\beta >0$ is small. 

\medskip
One of the first processes with sexual reproduction that was studied is Toom's NEC (north-east-center) rule. In its original formulation, see Toom (1974, 1980), the states of the sites were 1 and $-1$. Let $e_1$, $e_2$ be the two unit vectors. If the majority of the spins in $\{ x, x+e_1, x + e_2 \}$ was 1 at time $n$ then $\xi_{n+1}(x)=1$ with probability $1-p$ and 0 with probability $p$. If the majority of the spins  was $-1$ at time $n$ then $\xi_{n+1}(x)=-1$ with probability $1-q$ and 0 with probability $q$. If $p+q$ is small then the system has two stationary distributions, see e.g., Bennett and Grinstein (1985).

More relevant for us, is the reformulation of Toom's rule as a growth model:
\begin{align*}
1 \to 0 & \quad\hbox{rate 1} \\
0 \to 1 & \quad\hbox{rate $\lambda$ if $\xi_t(x+e_1)$ and $\xi_t(x+e_2)$ are both 1}
\end{align*}
If all the 1's in the initial configuration are inside a rectangle then there will never be any births outside the rectangle, so if we let
$$
\lambda_f = \inf \{ \lambda : P( \xi^A_t \neq \emptyset \hbox{ for all $t$ }) > 0
$$
be the critical value for survival from a finite set, then $\lambda_f=\infty$. 

Durrett and Gray (1985) used the contour method to prove, see announcement of results in Durrett (1985), $\lambda_c \le 110$.

\mn
If $p < p^* = 1 - $ the critical value for oriented bond percolation, then the process starting from product measure with density $p$ dies out. 

\mn
If $\lambda > \lambda_c$ and $\beta$ is such that $6 \beta^{1/4} \lambda^{3/4} < 1$ then when we add spontaneous births at rate $\beta$ there are two stationary distributions.

\medskip
Chen (1992, 1994) generalized Toom's growth model. He begins by defining  
\begin{center}
\begin{tabular}{cccc}
pair 1 & pair 2 & pair 3 & pair 4 \\
$x-e_1, x-e_2$ & $x+e_1, x-e_2$ & $x+e_1, x+e_2$ & $x-e_1, x+e_2$ 
\end{tabular}
\end{center}
The models are numbered by the pairs that can give birth: Type I (pair 1 = SWC); Type IV (any pair); Type III (pairs 1, 2, and 3); Type 2A (pairs 1 and 2); and Type 2B (pairs 1 and3).
Chen (1992) proved for Model IV that if $0 < p < p(\lambda)$ then
$$
P( 0 \in \xi^p_t ) \le t^{-c \log_{2\lambda}(1/p)}
$$
and for large $\lambda$
$$
\lim_{\beta\to 0} P( 0 \in \xi^{0,\beta}_\infty ) > 0
$$
so 0 is unstable under perturbation. In contrast, Chen (1994) shows that in model III 0 is stable under perturbation.  

Durrett and Neuhauser (1994) considered the behavior of the quadratic contact process, with stirring  (exchange of values at adjacent sites). They had deaths at rate 1, and births at rate $\beta$ times the fraction of adjacent pairs that are occupied. The mean field equation (which assumes adjaceent sites are independent) in this case:
$$
\frac{du}{dt} = - u + \beta (1-u)u^2
$$
has $\beta_c = 4$ and $\beta_f=\infty$. They showed that in the limit of fast stirring both critical values converged to 4.5. This threshold is the point where the pde 
$$
\frac{du}{dt} = u'' - u + \beta u (1-u)
$$ 
has traveling wave solutions $u(t,x)=w(x-ct)$ with $c>0$.  Based on simulations they conjectured that the phase transition was continuous. 

Evans, Guo, and Liu (in various permutations in five papers published in 2007--2009) have considered the quadratic contact process in which particles hop at rate $h$ (i.e., move according to the rules of the simple exclusion process, which for unlabeled particles is the same as sitrring). Birth rates are (1/4) times the number of adjacent pairs of occupied sites, deaths occur at rate $p$. Having $h>0$ means that $p_f(h)>0$. When $h < h_0$ is small $p_f(h) < p_e(h)$ the model has a discontinuous phase transition. and 0 is stable under perturbation. When $h \ge h_0$, $p_f(h) =p_e(h)$ and the phase transition is continuous 

The last three authors call their system Schl\"ogl's second model in honor of his (1972) paper which  introduced a model with a nonnegative integer number of particles per sites defined by the chemical reactions
$$
2X \rightleftharpoons 3X \qquad X \rightleftharpoons 0
$$
i.e., at a site with $k$ particles births occur at rate $c_0 + c_2 \binom{k}{2}$ and deaths occur at rate
$c_1 k + c_3 \binom{k}{3}$, and particles jump to a randomly chosen neighbor at rate $\nu$ each.
The system in which 
$$
X \rightleftharpoons 2X \qquad X \rightleftharpoons 0
$$
is Schl\"ogl's first model. It is the analogue of the linear contact process, or if you are a physicist, they are
in the same universality class.
Grassberger (1981) simulated a version of the second model in which the reaction $3X \to 2X$ was replaced
by the restriction of at most two particles per site, and in which doubly occupied sites give birth onto adjacent sites.
He found that this model has a second order (continuous) phase transition. See also Grassberger (1982),
which has been cited more than 300 times, or Prakash and Nicolis (1997) for a more recent treatment.

The threshold-$\theta$ contact process with $\theta \ge 1$ has been studied on $\ZZ^d$. Liggett (1994) used it and a comparison to show coexistence in a threhsold voter model. See also Chapter II.2 in Liggett (1999). Handjani (1997) studied the phase diagram of the model, while Mountford and Schonmann (2008) studied asymptotics for its critical values. However outside the physics literature, see da Silva and de Oliveira (2011), there are no results about the nature of the phase transition on $\ZZ^d$. As we explain later, Fontes and Schonmann (2008a) have considered the process on a tree.   

In this paper we will consider the discrete time threshold-$\theta$ contact process on a random $r$-regular graph, and on trees in which all vertices have degree $r$. In these processes, but sites that have at least two occupied neighbors at time $n$ are occupied with probability $p$ at time $n+1$. Our personal motivation, derived from participating in the 2010--2011 SAMSI program on Complex Networks, is that a random $r$-regular graph is a toy model for a
social network. This model, like the original small world graph of Watts and Strogatz (1998), is unrealistic because all vertices have
the same number of neighbors. We do not expect the qualitative behavior to change on an Erd\"os-Renyi graph, but this graph looks locally like a Galton-Watson tree so the proofs considerably more complicated. 

To see that properties of the model are sensitive to the degree distribution, recall that Chatterjee and Durrett (2009) have shown that if one studies the constact process on a random graph with a power law degree distribution, $p_k \sim C k^{-\alpha}$, then the critical value is 0 for any $\alpha < \infty$. It is an interesting question to determine whether or not the contact process
has positive critical value when the degree distribution has an exponential tail $p_k \sim C\exp(-\gamma k)$. Simulations of Chris Varghese suggest that the quadratic contact process on an Erd\"os-Renyi random graph have a discontinuous transition, but on the power-law graph in which $p_k = C k^{-2.5}$ for $k \ge 3$, $\lambda_c=0$ and the transition is continuous.

Our second motivation for exploring particle systems on a random $r$-regular graph is that it is the natural finite version of a
$r$-tree (in which each vertex has degree $r$). We think of a random regular graph as a ``tree torus", since the graph looks the same (in distribution) when viewed from any vertex. While the inspiration came from aesthetics, there is a practical consequence: the results for the random $r$-regular graph give as corollaries results for the threshold-$\theta$ contact process on a tree.

\subsection{Defining the process on the random graph}

In this paper, we study the behavior of the discrete time {\it
threshold-two contact process} on a random $r$-regular graph on $n$
vertices. We construct our random graph $G_n$ on the vertex set $V_n
:= \{ 1, 2, \ldots n \}$ by assigning $r$ ``half-edges" to each of
the vertices, and then pairing the half-edges at random. If $r$ is
odd, then $n$ must be even so that the number of half-edges, $rn$,
is even to have a valid degree sequence. Let $\bb P$ denote the
distribution of $G_n$, which is the first of several probability measures we
will define. We condition on the event $E_n$ that the
graph is simple, i.e., it does not contain a self-loop at any vertex,
or more than one edge between two vertices. It can be shown (see
e.g., Corollary 9.7 on page 239 of Janson, \L uczak and Rucin\'nski)
that $\bb P(E_n)$ converges to a positive limit as $n\to\infty$, and hence
 \beq \label{Ptilde}
\text{ if } \tilde{\bb P}:= \bb P(\cdot| E_n), \text{ then }
 \tilde{\bb P}(\cdot) \le c \bb P(\cdot) \text{ for some constant }
 c=c(r)>0.
\eeq
 So the conditioning on the event $E_n$ will not have much effect on the
distribution of $G_n$. It is easy to see that the distribution of $G_n$ under $\tilde{\bb P}$ is uniform over the
collection of all undirected $r$-regular graphs on the vertex set
$V_n$. We choose $G_n$ according to the distribution $\tilde{\bb P}$
on simple graphs, and once chosen the graph remains fixed through
time.

Having defined the graph, the next step is to define the dynamics on the graph.
We write $x\sim y$ to mean that $x$ is a neighbor of $y$, and let
\beq \label{N} {\cal N}_y:=\{x\in V_n: x\sim y\} \eeq be the set of
neighbors of $y$. The distribution $P_{G_n,p}$ of the (discrete
time) threshold-two contact process $\xi_t \subseteq V_n$ with
parameter $p$ conditioned on $G_n$ can be described as follows:
\begin{align*}
& P_{G_n,p}\left( x\in\xi_{t+1} \;  \left| \; |{\cal N}_x \cap \xi_t| \ge \theta \right.\right) =  p \text{ and } \\
& P_{G_n,p}\left( x\in\xi_{t+1}  \; \left| \; |{\cal N}_x \cap
\xi_t| < \theta \right.\right)  =  0,
\end{align*}
where the decisions for different vertices at time $t+1$ are taken
independently. Let $\xi_t^A \subseteq V_n$ denote the threshold-two contact process
starting from $\xi_0^A=A$, and let $\xi^1_t$ denote the special case
when $A=V_n$.

$P_{G_n,p}$ is sometimes called the {\it quenched measure}.
It is the distribution of the process conditioned on the graph.
The {\it annealed measure} $\mathbf P_p$ in which we average over the values of $G_n$ is defined by
$$
\mathbf P_p(\cdot) = \tilde{\bb E} P_{G_n,p} (\cdot),
$$
where $\tilde{\bb E}$ is the expectation corresponding to the
probability distribution $\tilde{\bb P}$.

\subsection{Main results}

The first step is to prove that threshold-2 contact process dies out for small $p$ values and survives
for $p$ close to 1. It is easy to see that on any graph in which all
vertices have degree $r$ the threshold-two contact process dies out
rapidly if $p<1/r$. An occupied site has at most $r$ neighbors that
it could cause to be occupied at the next time step, so
$E_{G_n,p}\xi^1_t \le n (rp)^t$. Our next result shows that if $r\ge
4$, $p$ is sufficiently close to 1, and if we start from all 1's, then with high probability the fraction of occupied
vertices in the threshold-two contact process $\ge 1-\ep_1$ for an
exponentially long time.

\begin{theorem}\label{p_c}
Suppose $\theta \ge 2$ and $r\ge \theta + 2$. There are constants $\ep_1, \gamma_1 > 0$,
and a good set of graphs ${\cal G}_n$ with $\tilde{\bb P}(G_n \in
{\cal G}_n ) \to 1$ so that if $G_n \in {\cal G}_n$ and $p \ge  p_1=
1 - \ep_1/(3r-3\theta)$, then
$$
P_{G_n,p}\left(\inf_{t\le \exp(\gamma_1n)}
\frac{|\xi^1_t|}{n} < 1-2\ep_1\right) \le \exp(-\gamma_1n).
$$
\end{theorem}

\noindent
Here and in what follows, all constants will depend on the degree $r$ and threshold $\theta$. If they depend on
other quantities this will be indicated.

The reason for the restriction to $r \ge \theta + 2$ comes from Proposition
\ref{isoper2}. When $r=\theta+1$ it is impossible to pick $\eta>0$ so that
$(1+\eta)/(r-\theta) < 1$. There may be more than algebra standing in the
way of constructing a proof. We conjecture that the result is false
when $r=\theta+1$. To explain our intuition in the special case $\theta=2$, consider a rooted binary tree
in which each vertex has two descendants and hence, except for the
root, has degree three. If we start with a density $u$ of 1's on
level $k$ and no 1's on levels $m<k$, then at the next step the
density will be $g(u)= pu^2 < u$ on level $k-1$. When there are
three descendants, then
$$
g(u) = p ( 3u^2 (1-u) + u^3 ),
$$
which has a nontrivial fixed point for $p \ge 8/9$ (divide by $u$ and solve the quadratic equation).

As the next result shows, there is a close relationship between the
threshold-$\theta$ contact process $\xi_t$ on a random $r$-regular graph
and the corresponding process $\zeta_t$ on the homogeneous $r$-tree.
Following the standard recipe for attractive interacting
particle systems, if we start with all sites on the tree occupied then
sequence of distributions decreases to a limit
$\zeta^1_\infty$, which is called the {\it upper invariant
measure}, since it is the stationary distribution with the most 1's.
The critical value is defined by
$$
p_c = \sup\{ p : P_p(\zeta^1_\infty(x) = 1 ) = 0 \},
$$

\begin{cor} \label{pctree}
Suppose $\theta \ge 2$, $r \ge \theta+2$ and that $p_1$ and $\ep_1$ are the constants
in Theorem \ref{p_c}. If $p \ge p_1$, then there is a translation
invariant stationary distribution for the threshold-two contact
process on the homogeneous $r$-tree in which each vertex is occupied
with probability $\ge 1-\ep_1$.
\end{cor}

\noindent
Fontes and Schonmann (2008a) have considered the continuous time threshold-$\theta$ contact process
on a tree in which each vertex has degree $b+1$ and have shown that if 
$b$ is large enough then $\lambda_c < \infty$. Our result improves their result by removing the restriction
that $b$ is large.

\subsubsection{Dying out from small density}

If we set the death rate = 0 in the threshold-$\theta$ contact process then the birth rate is
no longer important and the process reduces to bootstrap percolation. Balogh and Pittel (2007)
have studied bootstrap percolation on the random regular graph. They identified an interval
$[p_-(n),p_+(n)]$ so that the probability that all sites end up active goes sharply from
0 to 1. The limits $p_{\pm}(n) \to p_*$ and $p_+-p_-$ is of order $1/\sqrt{n}$.  
If bootstrap percolation cannot fill up the graph then it seems that process with deaths will
be doomed to extinction. The next result proves this, and more importantly extends the result
to arbitrary initial conditions with a small density of occupied sites.

Here, since processes with larger $\theta$ have fewer points, it is enough to prove the result
when $\theta=2$.

\begin{theorem}\label{th1}
Suppose $\theta\ge 2$ and $p_2<1$. There are constants $0< \ep_2
,C_2<\infty$ that depend on $p_2$, and a good set of
graphs ${\cal G}_n$ with $\tilde{\bb P}(G_n \in {\cal G}_n ) \to 1$
so that if $G_n \in {\cal G}_n$, then for any $p\le p_2$, and any
subset $A\subset V_n$ with $|A|\le \ep_2 n$,
$$
P_{G_n,p} \left( \xi^A_{C_2\log n} \ne \emptyset\right) \le
2/n^{1/6} \text{ for large enough $n$}.
$$
\end{theorem}

The density of 1's $\rho(p) = P_p(\zeta^1_\infty(x)=1)$ in the the
stationary distribution on the homogeneous $r$-tree is a
nondecreasing function of $p$. The next result shows that the
threshold-two contact process on the tree has a discontinuous phase
transition.

\begin{cor} \label{disco}
Suppose $\theta \ge 2$, let $p_1$ be the constant from Theorem \ref{p_c}, and let $\ep_2$ be the constant from Theorem \ref{th1}. 
$\rho(p)$ never takes values in $(0,\ep_2(p_1))$.
\end{cor}

\noindent
This result, like Theorem \ref{th1} does not require the assumption $r \ge \theta + 2$.
On the other hand if $\rho(p)\equiv 0$ for $r \le \theta +1$ the result is not very interesting in that case.
Again Fontes and Schonmann (2008a) have proved that the threshold-$\theta$ contact
process has a discontinuous transition when the degree $b+1$ is large enough.

Fontes and Schonmann (2008b) have studied $\theta$-bootstrap percolation on trees
in which each vertex has degree $b+1$ and $2\le \theta \le b$. They show that there is
a critical value $p_f$ so that if $p<p_f$ then for almost every initial configuration 
of product measure with density $p$, the final bootstrapped configuration does
not have any infinite components. This suggests that we might have $\ep_2(p)$
bounded away from 0 as $p \to 1$.

\subsubsection{Stability of 0}

The previous pair of results are the most difficult in the paper. From their proofs one easily gets
results for the process with sponataneous births with probability $\beta$, i.e., after the threshold-$\theta$
dynamics has been applied to the configuration at time $n$, we independently make vacant sites occupied with probability $\beta$.
Adding a superscript $0,\beta$ to denote the new process starting from all 0's, we have the following:

\begin{theorem}\label{th3}
Suppose $\theta \ge 2$. There is a good set of graphs ${\cal G}_n$ with $\tilde{\bb P}(G_n \in
{\cal G}_n ) \to 1$ so that for $\beta < \beta_3$, if $G_n \in {\cal G}_n$ and $p < 1$ then 
there are constants $C_3(p)$ and $\gamma_3(p,\beta)$ so that
$$
P_{G_n,p}\left(\sup_{t\le \exp(\gamma_1n)}
\frac{|\xi^{0,\beta}_t|}{n} > C_3\beta \right) \le 2\exp(-\gamma_3n).
$$
\end{theorem}

Let $\zeta^{0,\beta}_\infty$ be the limiting distribution for the process on the tree, which exists via
monotonicity. 

\begin{cor} \label{c3}
If $\theta \ge 2$ and $p<1$ then $\lim_{\beta\to 0} P( \zeta^{0,\beta}_\infty(x) = 1 ) = 0$.
\end{cor}

\subsection{Isoperimetric inequalities} \label{isopint}

We now describe the results that are the keys to the proof. 
Let $\partial U  :=  \{y\in U^c: y\sim x \text{ for some } x\in U\}$
be the boundary of $U$, and given two sets $U$ and $W$, let $e(U,W)$
be the number of edges having one end in $U$ and the other end in
$W$. Given an $x$ let $n_U(x)$ be the number of neighbors of $x$
that are in $U$ and let
$$
U^{*j} = \{ x \in V_n : n_U(x) \ge j \}.
$$
The keys to the proof of Theorem \ref{p_c} and \ref{th1} are ``isoperimetric inequalities"
that we will state in this section.

The estimation of the sizes of $e(U,U^c)$ is an enormous subject
with associated key words Cheeger's inequality and expander graphs.
Bollob\'as (1988) proved the following result for random regular
graphs:

\begin{theorem} Let $r \ge 3$ and $0<\eta < 1$ be such that
$$
2^{4/r} < (1-\eta)^{1-\eta} (1+\eta)^{1+\eta}.
$$
Then asymptotically almost surely a random $r$-regular graph has
$$
\min_{|U| \le n/2} \frac{e(U,U^c)}{|U|} \ge (1-\eta) r/2.
$$
\end{theorem}

\noindent To see that the constant is reasonable, choose $n/2$
vertices at random to make $U$. In this case we expect that $|e(U,
U^c)| = nr/2$.

While this result is nice, it is not really useful for us, because
we are interested in estimating the size of the boundaries $U^{*j}$
for $j\ge 2$, and in having better constants by only
considering small sets.

\begin{proposition} \label{isoper0}
Let $E^{*1}(m,\le k)$ be the event that there is a subset $U\subset
V_n$ with size $|U|=m$ so that $|U^{*1}| \le k$. There are constants
$C_0$ and $\Delta_0$ so that for any
$\eta>0$, there is an $\ep_0(\eta)$ which also depends on $r$ so
that for $m \le \ep_0(\eta)n$,
$$
\bb P\left[E^{*1}(m,\le (r-1-\eta)m)\right] \le
C_0\exp\left(-\frac{\eta^2}{4r} m \log(n/m) +\Delta_0m \right).
$$
\end{proposition}

This result yields the two results we need to prove Theorems
\ref{p_c} and \ref{th1}. To obtain the first, note that if $W=V_n
\setminus \xi_t$ is the set of vacant vertices at time $t$, and $j=\theta - 1$ then
$W^{*(r-j)}$ is the set of vertices which will certainly be vacant
at time $t+1$. From the definitions it is easy to see that if
$|W|=m$, then
$$
rm \ge \sum_{y\in W^{*1}} e(\{ y\},W) \ge |W^{*1} \setminus W^{*(r-j)}|
+(r-j)|W^{*(r-j)}| = |W^{*1}| + (r-j-1)|W^{*(r-j)}|.
$$
So for any set $W$ of size $m$, if $|W^{*(r-j)}| \ge k$, then
$|W^{*1}| \le rm-(r-j-1)k$. Taking $k=m(1+\eta)/(r-j-1)$ so that
$rm - (r-j-1) k = (r-1-\eta)m$ and using Proposition \ref{isoper0} we get

\begin{proposition} \label{isoper2}
Let $E^{*(r-j)}(m,>k)$ be the event that there is a subset $W\subset
V_n$ with size $|W|=m$ so that $|W^{*(r-j)}| > k$. For the constants
$C_0$, $\Delta_0$, and $\ep_0(\eta)$ given in Proposition
\ref{isoper0} and $m \le \ep_0(\eta)n$,
$$
\bb P\left[E^{*(r-j)}\left(m,>\left( \frac{1+\eta}{r-j-1} \right)m
\right)\right] \le C_0\exp\left(-\frac{\eta^2}{4r} m \log(n/m)
+\Delta_0m \right).
$$
\end{proposition}

To derive our second result, we note that if $|U|=m$, then
$$
rm \ge \sum_{y\in U^{*1}} e(\{ y\},U) \ge |U^{*1} \setminus U^{*2}|+2|U^{*2}| = |U^{*1}| + |U^{*2}|.
$$
So for any set $U$ of  size $m$, if $|U^{*2}| \ge k$, then $|U^{*1}| \le rm-k$.
Taking $k=(1+\eta)m$ we get

\begin{proposition} \label{isoper1}
Let $E^{*2}(m,> k)$ be the event that there is a subset $U\subset
V_n$ with size $|U|=m$ so that $|U^{*2}| > k$. For any $\eta>0$ and
the constants $C_0$, $\Delta_0$, and $\ep_0(\eta)$ given in
Proposition \ref{isoper0} if  $m \le \ep_0(\eta)n$, then
$$
\bb P\left[E^{*2}(m,> (1+\eta)m)\right] \le
C_0\exp\left(-\frac{\eta^2}{4r} m \log(n/m) +\Delta_0m \right).
$$
\end{proposition}

\section{Upper bound on the critical value $p_c$}\label{p_csec}

\begin{proof} [Proof of Theorem \ref{p_c}] Recall that $r \ge \theta+2$.
Let $\eta=1/3$ and let $\ep_1= \exp(-8\Delta_0 r/\eta^2)$ so that if $m=[\ep_1 n]$, then $\log(n/m) = 8 \Delta_0 r /\eta^2$
and hence $(\eta^2/4r) \log(n/[\ep_1 n]) = \Delta_0$.
With these choices, Proposition \ref{isoper2} implies
$$ 
\bb P\left[E^{*(r-j)}\left([\ep_1 n] ,> \frac{4[\ep_1 n]}{3(r-j-1)}\right)\right]  
\le C_0\exp\left(- \Delta_0  \ep_1 n \right)
$$
Let ${\cal G}_n = E^{*(r-j)}(m,\le(1+\eta)m/(r-j-1))$. Since
increasing the size of a set $U$ increases $U^{*\theta}$, it follows that
if $G_n \in {\cal G}_n$ and $|U| \ge (1 - \ep_1) n$, then
$$
|U^{*\theta}| \ge \left(1 - \frac{4\ep_1}{3(r-\theta)}  \right) n.
$$
If $|\xi_t| \ge (1-\ep_1)n$ and $p \ge 1 - \ep_1/(3r-3\theta)$,
then the distribution of $|\xi_{t+1}|$ dominates a
$$
\hbox{Binomial}\left( \left(1 - \frac{4\ep_1}{3(r-\theta)} \right) n, p
\right) \hbox{ distribution, which has mean }
\ge \left(1- \frac{5\ep_1}{3(r-\theta)} \right) n.
$$
(the $\ep_1^2$ term is $>0$).
When $r \ge \theta+2$, this is $> (1-\ep_1)n$ so standard large deviations
for the Binomial distribution imply that there is a constant
$\gamma_1(r)>0$ so that
$$
P_{G_n,p}\Bigl( |\xi_{t+1}| < (1-\ep_1)n \Bigl| |\xi_t| \ge
(1-\ep_1) n \Bigr) \le \exp(-2\gamma_1 n).
$$
If we set $T=\lfloor \exp(\gamma_1 n)\rfloor$, then the probability
that $|\xi_{t+1}| \ge (1-\ep_1)n$ fails for some $t\le T$ is $\le
\exp(-\gamma_1 n)$, which completes the proof of Theorem \ref{p_c}.
\end{proof}

To prepare for the proof of the Corollary we need the following
result which shows that the random regular graph looks locally like
a tree. See e.g., Lemma 2.1 in Lubetsky and Sly (2010).

\begin{lemma} \label{treelike}
Suppose $r \ge 3$ and let $R = (1/3) \log_{r-1} n$. For any $x$ the probability that the
collection of points within distance $R$ of $x$ differs from the
homogeneous $r$-tree is  $\le 10 n^{-1/3}$ for large $n$.
\end{lemma}

\begin{proof} Starting with $x$ its neighbors are chosen by selecting $r$ half edges at random from the $rn$ possible options. This procedure continues to select the neighbors of the neighbors, etc. To generate all of the connections out to distance $R$
we will make 
$$
r[1+(r-1)+ \cdots + (r-1)^{R-1}] \le rn^{1/3}/(r-2) \quad\hbox{choices.}
$$
The probability that at some point we select a vertex that has already been touched is
$$ 
\le \frac{rn^{1/3}}{r-2} \cdot \frac{rn^{1/3}/(r-2)}{n-rn^{1/3}/(r-2)} \le 10n^{-1/3}
$$
for large $n$.
\end{proof}

\begin{proof} [Proof of Corollary \ref{pctree}]
Let $r\ge \theta+2$, $p \ge p_1$, and $t(n) = \log\log n$. To prove that
the upper invariant measure is nontrivial we will show that
$\lim_{n\to\infty} P_{p_1}( \zeta^1_{t(n)}(0)=1 ) \ge 1-\ep_1$. To
do this note that Lemma \ref{treelike} together with a standard
second moment applied to $H_n =$ the number of vertices whose neighbors are tree-like up to distance $\log \log n$
implies that $\tilde{\bb P} ( H_n  \le (1-\ep_1) n) \to 0$.  So
we can choose $G_n\in{\cal G}_n$ having this property. For such
a $G_n$ Theorem \ref{p_c} implies
$$
\liminf_{n\to\infty} \frac{1}{n} \sum_{x=1}^n P_{G_n,p_1}
(\xi^1_{t(n)}(x) = 1) \ge 1-\ep_1.
$$
Now the state of $x$ at time $t(n)$ can be determined by looking at
the values of the process on the space-time cone $\{ (y,s) : d(x,y)
\le t(n)-s \}$. Since the space-time cones corresponding to $n-o(n)$
many points of $G_n$ are same as that corresponding to 0 of the
homogeneous $r$-tree, the desired result follows.
\end{proof}

\section{Extinction from small density, stability of 0}

\begin{proof}[Proof of Theorem \ref{th1}]
Pick $\eta=\eta(p_2)>0$ so that $(1+\eta)p_2 = (1-3\eta)$ and then
pick $\ep_2 = \exp(-8\Delta_0 r/\eta^2)$ so that for $m\le \ep_2 n$
we have
$$
\bb P[E^{*2}(m ,>(1+\eta)m ] \le C_0\exp\left(-\frac{\eta^2}{8r} m
\log(n/m) \right).
$$
Let ${\cal G}_n$ be the collection of graphs so that $E^{*2}(m ,\le(1+\eta)m)$ holds for all $m\le \ep_2 n$.
To see that when $n$ is large this event has high probability, note that
\begin{align*}
\bb P( {\cal G}_n^c )  & \le \sum_{m=n^{a}+1}^{\ep_2 n}
C_0 \exp\left(-\frac{\eta^2}{8r} n^{a}\log(1/\ep_2)\right)
 + \sum^{n^{a}}_{m=1} C_0 \exp\left(-\frac{\eta^2}{8r} \log(n^{1-a})\right) \\
& \le C_0 \ep_2 n \exp\left(-\frac{\eta^2}{8r} n^{a}\log(1/\ep_2)\right)
+ C_ n^a n^{-\eta^2(1-a)/8r} \to 0
\end{align*}
if $a$ is chosen small enough.

As in the proof of Theorem \ref{p_c}, we will use large deviations
for the Binomial distribution to control the behavior of the
process. However, this time the value of $p$ changes with $m$, and
we will have to stop when the set gets too small. According to Lemma
2.8.5 in Durrett (2007)

\begin{lemma}
If $X=\hbox{Binomial}(k,q)$, then
$$
P( X \ge k(q+z) ) \le \exp(-kz^2/2(q+z)).
$$
\end{lemma}

\noindent Using this result with $k=(1+\eta)m$ and $q=p_2$ which
have $kq=(1-3\eta)m$, then taking $z = \eta < 2\eta/(1+\eta)$ so
that $k(q+z) \le (1-\eta)m$, it follows that for $m \le \ep_2 n$ and
$G_n \in {\cal G}_n$, 
\beq 
P_{G_n,p}\Bigl( |\xi_{t+1}| > (1-\eta) m
\Bigl| |\xi_t| = m \Bigr) \le \exp( - \eta^2 m / 2(p_2+\eta) ).
\label{expdie} \eeq 
Using this result $\ell = \lceil(1/2) \log
n/(-\log(1-\eta))\rceil$ times, we see that if $|\xi_0| \le \ep_2 n$
and $\nu:=\inf\{t: |\xi_t| \le n^{1/2}\}$, then with high
probability $\nu \le \ell$.

To finish the process off now we note that when $m \le \ep_2 n$,
 \beq \label{Expbd}
 E_{G_n,p}\Bigl( |\xi_{t+1}| \Bigl| |\xi_t|=m \Bigr) \le (1-3\eta)m.
 \eeq
 Also note that if $\kappa = \lceil(2/3) \log
n/(-\log(1-3\eta))\rceil$ and $G_n \in {\cal G}_n$, then
 $$ |\xi_{\nu+t}| \le (1+\eta)^t n^{1/2} \le n^{5/6} \text{ for }
 1 \le t\le \kappa, \text{ as } (1+\eta)^2(1-3\eta) < 1 \text{ for any $\eta>0$.}$$
 So using the inequality in \eqref{Expbd} $\kappa$ times
 we have $P_{G_n, p}(|\xi_{\nu+\kappa}| \ge 1) \le 1/n^{1/6}$. So combining with
 \eqref{expdie} we conclude that if $|\xi_0| \le \ep_2 n$ and $G_n \in {\cal G}_n$, then
$$
P_{G_n,p}(|\xi_{\kappa+\ell}| \ge 1 ) \le 2/n^{1/6} \text{ for large
enough $n$},
$$
which proves the desired result with $C_2 = 2/(-\log(1-\eta))$.
\end{proof}

\begin{proof}[Proof of Corollary \ref{disco}]
Suppose that the upper invariant measure for the process on the
homogeneous $r$-tree has $\rho(p) \in (0,(1-3\delta)\ep_2(p_1,r))$
for some $\delta>0$. It is easy to see that $\ep_2(p_1,r) <
1-\ep_1$, and so it follows from Corollary \ref{pctree} that
$p<p_1$. Pick a time $\tau$ so that the threshold-two contact
processes on the homogeneous $r$-tree has $P_p( \zeta^1_\tau(0) = 1)
< (1-2\delta)\ep_2(p_1,r)$. The argument involving Lemma
\ref{treelike} in the proof of Corollary \ref{pctree} can be
repeated to see that we can choose $G_n\in {\cal G}_n$ so that the
neighborhoods of $n-o(n)$ many points of $G_n$ within distance
$\tau$ look exactly like the analogous neighborhood of 0 in the
homogeneous $r$-tree. If $n$ is large, then for the above choices of
$\tau$ and $G_n$,
$$
\frac{1}{n} \sum_{x=1}^n P_{G_n,p} (\xi^1_{\tau}(x) = 1) \le
(1-\delta) \ep_2(p_1,r).
$$
Since points on $G_n$ separated by more than $2\tau$ are independent
in $\xi^1_\tau$, it follows that with $P_{G_n,p}$-probability
tending to 1 as $n \to \infty$
$$
\sum_{x=1}^n \xi^1_{\tau}(x) \le \ep_2(p_1,r) n.
$$
Formula (\ref{expdie}) implies that after $\ell = \lceil (\log \log
n)/(-\log(1-\eta(p_1)))\rceil$ time units
$$
P_{G_n,p}\left(\sum_{x=1}^n \xi^1_{\tau+\ell}(x) \le  n/\log
n\right) \to 1 \text{ as $n\to\infty$}.
$$
So by our choice of $G_n$ we have $P_p( \zeta^1_{\tau+\ell}(0) = 1)
\le \rho(p)/2$. Since by monotonicity $P_p(\zeta^1_t(0)=1)$ is a
decreasing function of $t$, we get  a contradiction that proves the
desired result.
\end{proof}

\begin{proof}[Proof of Theorem \ref{th3}]
From (\ref{expdie}) and a standard large deviations result for the Binomial,
it follows that there is a constant $\delta_1(p,\beta)$ for $m \le \ep_2 n$ and $G_n \in {\cal G}_n$, so that
\beq 
P_{G_n,p}\Bigl( |\xi^{0,\beta}_{t+1}| > (1-\eta) m + 2\beta (n-m) )
\Bigl| |\xi^{0,\beta}_t| = m \Bigr) \le \exp( - 2\delta m  ).
\eeq 
Let $\bar m = (1-\eta-2\beta) m + 2\beta n$, and $\alpha = 2\beta/(\eta+2\beta)$. 
If $m = \alpha n$ then $\bar m = m$, while if $m \ge 2\alpha n$
$$
\bar m - \alpha n = (1-\eta-2\beta) (m-\alpha n)
$$
and hence
$$
\bar m \le \left( 1 - \frac{\eta+2\beta}{2} \right) m.
$$   
The probability for the number of particles to jump from $m = 2 \alpha n$ 
to $\ge 3 \alpha n$ is $\le \exp(-\delta_2(p,\beta)n)$, so by monotonicity this is true for $m\le 2 \alpha n$.
The desired result with $C_3 = 3 \alpha$ and $\gamma_3 = \min\{\delta_1,\delta_2\}$ follows easily from these observations.
\end{proof}

\begin{proof}[Proof of Corollary \ref{c3}] 
Suppose that $P( \zeta^{0,\beta}_\infty(x) = 1 ) \ge 5 \alpha$. If so then there is
a time $\tau$ at which $P( \zeta^{0,\beta}_\infty(x) = 1 ) \ge 4 \alpha$.
The argument involving Lemma
\ref{treelike} in the proof of Corollary \ref{pctree} can be
repeated to see that we can choose $G_n\in {\cal G}_n$ so that the
neighborhoods of $(1-\alpha)n$ many points of $G_n$ within distance
$\tau$ look exactly like the analogous neighborhood of 0 in the
homogeneous $r$-tree. But them we have a contradiction with the result in
Theorem \ref{th3}.
\end{proof}

\section{Estimates for $e(U,U^c)$ and $|\partial U|$} \label{lennas}

We begin with a simple estimate for the number of subsets of $V_n$
of size $m$.
\begin{lemma}\label{subsetsizebd}
 The number of subsets of $V_n$ of size $m$ is at most
 $\exp(m\log(n/m)+m)$.
 \end{lemma}

\noindent
{\it Proof.} The number of subsets of $V_n$ of size $m$ is ${n\choose m}$. Using
$n(n-1)\cdots(n-m+1) \le n^m$ and $e^m > m^m/m!$,
 $$
{n\choose m} \le \frac{n^m}{m!} \le \left(\frac{ne}{m}\right)^m = \exp(m\log(n/m)+m).
\eqno\square
$$

 In order to study the distribution of $|\partial U|$, the first step is to estimate
 $e(U,U^c)$. Because of the symmetries of our random graph $G_n$, the distribution of $e(U,U^c)$ under $\bb P$ depends on $U$ only through $|U|$.

\begin{lemma}\label{crossedge}
There are numerical constants $C_{\ref{crossedge}}$ and
$\Delta_1=r(2 + 1/e)+3/2$ so that if $U$ is a subset of $V_n$ with
$|U|=m$ and $\alpha\in[0,1]$, then
 $$
\bb P(e(U,U^c) \le \alpha r|U|) \le
C_{\ref{crossedge}}\exp\left(-\frac{r}{2}(1-\alpha)m\log(n/m) +
\Delta_1 m \right).
$$
\end{lemma}

\begin{proof} Let $f(u)$ be the number of ways of pairing $u$ objects. Then
Stirling's formula $n! \sim (n/e)^{n} \sqrt{2\pi n}$ implies
$$
f(u)=\frac{u!}{(u/2)!2^{u/2}} \sim \sqrt{2} (u/e)^{u/2},
$$
and it follows from the limit result that $C_1 (u/e)^{u/2} \le f(u) \le  C_2 (u/e)^{u/2}$
for all integers $u$.

If $q(m,s)= \bb P(e(U,U^c)=s)$, then we have
$$
q(m,s)\le  {rm\choose s} {r(n-m)\choose s} s! \frac{f(rm-s)f(r(n-m)-s)}{f(rn)}.
$$
 To see this, recall that we construct the graph $G_n$ by pairing the
 half-edges at random, which can be done in $f(rn)$ many ways as
 there are $rn$ half-edges. We can choose the left endpoints of
 the edges from $U$ in ${rm\choose s}$ many ways, the right endpoints
 from $U^c$ in ${r(n-m)\choose s}$ many ways, and pair them in $s!$ many ways. The
 remaining $(rm-s)$ many half-edges of $U$ can be paired among
 themselves in $f(rm-s)$ many ways. Similarly the remaining
 $(r(n-m)-s)$ many half-edges of $U^c$ can be paired among
 themselves in $f(r(n-m)-s)$ many ways.

To bound $q(m,s)$ we will use an argument from Durrett (2007) that
begins on the bottom of page 161 and we will follow it until the
last display before (6.3.6). To make the connection we note that
their $p(m,s) = \binom{n}{m} q(m,s)$ and write $D=rn,
k=rm$ and $s=\eta k$ for $\eta \in [0,1]$ to get
 \beq\label{eq1}
q(m,s)\le C k^{1/2} \left(\frac{e^2}{\eta}\right)^{\eta k}
 \left(\frac{k}{D}\right)^{k(1-\eta)/2}
 \left(1-\frac{(1+\eta)k}{D}\right)^{(D-(1+\eta)k)/2}.
\eeq
A little calculus gives
 \beq\label{phi} \text{if } \phi(\eta) = \eta\log(1/\eta), \text{ then } \phi'(\eta)=-(1+\log\eta) \text{ and }
\phi''(\eta) = -\frac 1\eta.
\eeq
So $\phi(\cdot)$ is a concave function and its derivative vanishes at $1/e$.
This shows that the function $\phi(\cdot)$
is maximized at $1/e$, and hence
\beq
0 \le \eta \log(1/\eta) \le 1/e \quad\hbox{ for $\eta \in [0,1]$}.
\label{xlogxbd}
\eeq
So $(e^2/\eta)^{\eta k}\le B^k$ with $B=exp(2+1/e)$. If we ignore the last term of \eqref{eq1},
which is $\le 1$, then we have
 \begin{align*}
 \bb P(e(U,U^c) \le \alpha rm)
\le  & \sum_{\{\eta: \; \eta rm \in\bb{N},\, \eta \le \alpha\}}
 C(rm)^{1/2} B^{rm} \left(\frac mn\right)^{rm(1-\eta)/2}\\
& \quad  \le  C r^{3/2} m^{3/2} B^{rm} \left(\frac mn\right)^{r(1-\alpha)m/2},
 \end{align*}
 as there are at most $rm$ terms in the sum and $(m/n)^{1-\eta}\le
 (m/n)^{1-\alpha}$ for $\eta\le\alpha$. The above
 bound is
$$
\le C\exp\left(-\frac{r}{2}(1-\alpha)m \log(n/m) + rm\log B + 3m/2\right)
$$
and we have the desired result.
\end{proof}

Lemma \ref{crossedge} gives an upper bound for the probability that
$e(U,U^c)$ is small. Our next goal is to estimate the
difference between $e(U,U^c)$ and $|\partial U|$.

\begin{lemma}\label{dampen}
If $U$ is a subset of vertices of $G_n$ such that $|U|=m$, then
there is a constant $\Delta_2$ that depends only on $r$ and an
$\ep_{4.3}(\eta)$ which also depends on $r$ so that for any $0 <
\eta \le u \le r$, and $m \le \ep_{4.3}(\eta)n$,
$$
\bb P\left(\left.|\partial U|\le (u-\eta) |U| \; \right|  \, e(U,U^c)=u|U|\right)
\le  \exp(-\eta m\log(n/m)  + \Delta_2m).
$$
\end{lemma}

\begin{proof} To construct $e(U,U^c)$ we choose $um$ times from the set of $r(n-m)$ half edges
attached to $U^c$. We want to show that with high probability at
least $(u-\eta)m$ vertices of $U^c$ are touched. To do this it is
enough to show that if the half-edges are chosen one by one, then
with high probability at most $\eta m$ of them are attached to one
of the already touched vertices. We will call the selection of
half-edge associated with a vertex that has already been touched a
bad choice. At any stage in the process there are at most $(r-1)um$
bad choices among at least $r(n-m)-um$ choices. Thus the number of
bad choices is stochastically dominated by
$$
X \sim \hbox{Binomial}\left( N = um, p = \frac{(r-1)um}{r(n-m)-um}
\right).
$$
When $u \le r$ and $m\le n/3$, we have $r(n-m) - um \ge r(n-2m) \ge rn/3$ and hence
$$
p \le \frac{(r-1)u}{r/3} \cdot \frac{m}{n} \le \frac{\eta}{u}
$$
when $m \le \ep_{4.3}(\eta)n$.

A standard large deviations result for the Binomial distribution,
see e.g., Lemma 2.8.4 in Durrett (2007) implies $P( X \ge Nc ) \le
\exp(-NH(c))$ for $c>p$, where \beq H(c) = c \log\left( \frac{c}{p}
\right) + (1-c) \log \left( \frac{1-c}{1-p} \right). \label{Bldpb}
\eeq When $c=\eta/u$, the first term in the large deviations bound
\eqref{Bldpb}
\begin{align*}
\exp(-Nc\log(c/p)) & \le \exp\left( - um \cdot \frac{\eta}{u} \left[
\log(n/m) + \log(\eta)
+ \log\left(\frac{r/3}{u^2(r-1)} \right) \right] \right) \\
& \le \exp[ - \eta m \log (n/m) + (m/e) + m \eta  \log(3r(r-1)) ]
\end{align*}
by \eqref{xlogxbd}. For the second term in the large deviations
bound \eqref{Bldpb} we note that $1/(1-p) > 1$ and use
\eqref{xlogxbd} to conclude
$$
\exp\left(-N(1-c)\log \left( \frac{1-c}{1-p} \right) \right) \le
\exp\left(-N(1-c)\log(1-c) \right) \le \exp(um/e),
$$
which proves the desired result for
$\Delta_2=(r+1)/e+r\log(3r(r-1))$.
\end{proof}

\section{Proof of Proposition \ref{isoper0}} \label{isoperpf}

We begin by recalling some definitions given earlier and make two
new ones. Let $\partial U  :=  \{y\in U^c: y\sim x \text{ for some }
x\in U\}$ be the boundary of $U$, and given disjoint sets $U$ and
$W$ let $e(U,W)$ be the number of edges between $U$ and $W$. Given a
vertex $x$, let $n_U(x)$ be the number of neighbors of $x$ that are
in $U$ and let $U^{*1} = \{ x \in V_n : n_U(x) \ge 1 \}$. Let $U_0 =
\{ x \in U: n_U(x)=0 \}$ be the set of isolated vertices in $U$, and
let $U_1 = U - U_0$.

\begin{proof}
Given $\eta>0$ define the following events:
\begin{align}
A_U  & = \{|U_1|\ge (\eta/2r)|U|\},
\nonumber \\
B_U  &=  \{|U^{*1}|\le(r-1-\eta)|U|\},
\label{Uevents} \\
D_U  &=  \{e(U,U^c) \le (r-2-\eta)|U|\}. \nonumber
\end{align}
There are three steps in the proof.

\ms I. Estimate the probability of $F_1 = \cup_{\{U : |U|=m\}}
\left(B_{U}\cap A_{U}^c\right).$

\ms II. Estimate the probability of $ F_2 = \cup_{\{W: (\eta/2r) m
\le |W| \le  m\}} D_{W}.$

\ms III. Estimate the probability of $F_3 = \cup_{\{U : |U|= m\}}
B_U \cap F_1^c \cap F_2^c$.

\mn
{\it Step I:}  On the event $A_U^c$, $|U_0|>(1-\eta/2r)|U|$ and so $e(U,U^c)\ge r|U_0|\ge
 (r-\eta/2)|U|$. Also on the event $B_U$, $|\partial U| \le |U^{*1}| \le
 (r-1-\eta)|U|$. From these two observations we have
\begin{align}
 \bb P(B_U\cap A_U^c)
 & \le \bb P( |\partial U|\le (r-1-\eta)|U| ,  e(U,U^c)\ge (r-\eta/2)|U| ) \notag \\
 & \le \bb P(e(U,U^c)-|\partial U| \ge (1+\eta/2)|U|). \label{BUAUcbd}
\end{align}
Combining  \eqref{BUAUcbd} with the bound in Lemma \ref{dampen}, we
see that if $|U|=m \le \ep_{4.3}(1+\eta/2)n$, then
\beq  \label{BUAUc}
\bb P(B_U \cap A_U^c)  \le  \exp\left[-(1+\eta/2)m \log(n/m) + \Delta_2 m\right].
\eeq
Using \eqref{BUAUc} and the inequality in Lemma
\ref{subsetsizebd} if $m \le \ep_5 n$, then
\begin{align}
  \bb P(F_1) & \le   {n\choose m} \exp\left[-(1+\eta/2) m \log(n/m) + \Delta_2 m\right] \notag \\
 & \le \exp\left[-(\eta/2) m \log(n/m) +(1+ \Delta_2) m\right].\label{F1bound}
\end{align}
 If $m$ is small enough, then the above estimate is exponentially small, and so with
 high probability there is no subset $U$ of size $m$ for which $B_U
 \cap A_U^c$ occurs.

\mn
 {\it Step II:} Our next step is to estimate the probability that there is a set $U$ of size $m$ for which
 $A_U$ occurs and $e(U_1,U_1^c) \le (r-2-\eta) |U_1|$. If $A_U$ occurs for some subset $U$ of size $m$,
 then $|U_1| \in [\eta m/2r , m]$. Using Lemma \ref{crossedge} with $\alpha=1-(2+\eta)/r$ and
the inequality in Lemma \ref{subsetsizebd},
\begin{align}
 \bb P(F_2) & =  \bb P\left(\cup_{m'\in[\eta
 m/2r , m]} \cup_{\{W:|W|=m'\}} \{e(W,W^c) \le (r-2-\eta) m'\}\right) \notag\\
 & \le \sum_{m' \in [\eta   m/2r , m]} {n\choose m'}
 C_{4.2} \exp\left[-\left(\frac{2+\eta}{2}\right)m'\log(n/m') + \Delta_1 m'\right] \notag\\
 & \le  \sum_{m' \in [\eta m/2r , m]}
C_{4.2} \exp\left(-(\eta/2)m'\log(n/m')+(1+\Delta_1)m'\right).
\label{F2bound1}
 \end{align}
 The function $\phi(\eta)=\eta\log(1/\eta)$ is increasing on
 $(0,1/e)$ (see \eqref{phi}), so if $m \le n/e$ and $m'\in [\eta m/2r, m]$,
$$
m'\log(n/m') \ge (\eta m/2r)\log(2rn/\eta m) \ge (\eta /2r)m
\log(n/m),
$$
since $(\eta/2r)\log(2r/\eta) >0$. Using the facts that there are fewer than $m$ terms in the
sum over $m'$ and the inequality $m \le e^m$ for $m \ge 0$, we have
 \beq\label{F2bound}
 \bb P(F_2) \le C_{4.2}\exp\left(-(\eta^2/4r) m \log(n/ m)+ (2+\Delta_1) m\right).
\eeq when $m \le n/e$.  If $m$ is small enough, then the right-hand
side of \eqref{F2bound} is exponentially small, and so with high
probability there is no subset $U$ of size $m$ for which $A_U$
occurs and $e(U_1,U_1^c) \le (r-2-\eta) |U_1|$.

\mn {\it Step III: } Noting that $U^{*1}$ is a disjoint union of
$U_1$ and $\partial U$ we see that if $B_U$ occurs, then
$$
(r-1-\eta) |U| \ge |U^{*1}| = |U_1| + |\partial U|.
$$
Using $|U| = |U_0|+|U_1|$ now we have
\beq \label{Delta}
 \Delta(U) \equiv |\partial U| - (r-2-\eta)|U_1| -  (r-1-\eta)|U_0| \le
 0.
\eeq
Also  if $|U|=m$, then by the definition of $F_1$,  $B_U \cap F_1^c  \subset B_U \cap A_U $,
 and on the event $A_U \cap F_2^c$, we have $|U_1|\ge (\eta /2r)|U|$ and $e(U_1,U_1^c) > (r-2-\eta)|U_1|$.
Combining these observations,
 \begin{align} \label{step3eq1}
\bb P(B_U \cap F_1^c \cap F_2^c) &\le \bb P(B_U \cap A_U \cap F_2^c) \\
&\le \bb  P( \Delta(U)\le 0 , e(U_1,U_1^c) > (r-2-\eta) |U_1| ). \nonumber
\end{align}
 Now by the definitions of $U_0$ and $U_1$, we have $e(U_0,U^c)=r|U_0|$ and hence
\beq
e(U,U^c)=r|U_0|+e(U_1,U_1^c),
\label{crossedgebreakup}
\eeq
and a little algebra shows that
$$
\{\Delta(U) \le 0\}=  \{e(U,U^c) - |\partial U| \ge (1+\eta)|U_0| + e(U_1,U_1^c) - (r-2-\eta)|U_1|\}.
$$
Also $e(U_1,U_1^c) < r|U_1|$. So
\begin{align}
 & \bb  P( \Delta(U)\le 0 , e(U_1,U_1^c)  > (r-2-\eta) |U_1| ) \label{step3eq2}\\
 & =  \sum_{\gamma \in (0 , 2+\eta)}
\bb P\left(  e(U_1,U_1^c)=(r-2-\eta+\gamma)|U_1| , e(U,U^c) - |\partial U| \ge (1+\eta)|U_0|+\gamma|U_1|
\right).\notag
\end{align}
Combining \eqref{step3eq1} and \eqref{step3eq2}, and recalling that  $|U_1| \in [\eta m/2r , m]$,
\beq
\bb P(B_U \cap F_1^c \cap F_2^c) = \sum_{\gamma \in (0,2+\eta)} \; \;
 \sum_{k \in [\eta m/2r , m]} \bb P(G_{\gamma,k}) \bb P( H_{\gamma} | G_{\gamma,k}
 ),
\label{step3eq3}
\eeq
where $G_{\gamma , k} =  \{e(U_1,U_1^c)=(r-2-\eta+\gamma)|U_1| , |U_1|=k\}$ and
$$
H_{\gamma}  = \{e(U,U^c) - |\partial U|\ge (1+\eta)|U_0|+\gamma|U_1| \}.
$$
In view of \eqref{crossedgebreakup}, if $R=r-2-\eta$ and
$L=(R+\gamma)k+r(m-k)$, then $e(U,U^c)=L$ on $G_{\gamma, k}$. So
$$
\bb P( H_{\gamma} | G_{\gamma,k} ) = \bb P(e(U,U^c)-|\partial U| \ge \gamma
 k+(1+\eta)(m-k) \; | \, e(U,U^c)=L,|U_1|=k).
$$
Since under the conditional distribution $\bb P(\cdot | e(U,U^c)=L)$
all the size-$L$ subsets of half-edges corresponding to $U^c$ are
equally likely to be paired with those corresponding to $U$, the
conditional distribution of $e(U,U^c)-|\partial U|$ given $e(U,U^c)$
and $|U_1|$ does not depend on $|U_1|$. So we can drop the event
$\{|U_1|=k\}$ from the last display and use Lemma \ref{dampen} with
$\eta$ replaced by $\eta' = (\gamma k + (1+\eta)(m-k))/m$ to see
that if $m \le \ep_{4.3}(\eta')n$, then
 \beq \label{s}
\bb P( H_{\gamma} | G_{\gamma,k} ) \le \exp\left(-\{\gamma
 k+(1+\eta)(m-k)\} \log(n/m) + \Delta_2 m\right).
\eeq

In order to estimate $\bb P(G_{\gamma,k})$, we again use \eqref{crossedgebreakup} to conclude
\begin{align*}
\bb P(G_{\gamma,k}) & = \bb P(e(U_1,U_1^c)=(r-2-\eta+\gamma)k , |U_1|=k)\\
 & = \bb P(e(U,U^c)=(r-2-\eta+\gamma)k+r(m-k) , |U_1|=k)\\
 & \le \bb P(e(U,U^c)=rm-(2+\eta-\gamma)k),
\end{align*}
Using Lemma \ref{crossedge} with $\alpha=1-(2+\eta-\gamma)k/rm$,
 \beq \label{r}
\bb P(G_{\gamma,k}) \le C_{4.2}\exp\left(-\frac{2+\eta-\gamma}{2} k
\log(n/m) + \Delta_1 m\right). \eeq Combining \eqref{step3eq3},
\eqref{s} and \eqref{r} if $m \le \ep_0(\eta) n$, where $\ep_0(\eta)
= \min\{ \ep_{4.3}(1+\eta/2), \ep_{4.3}(\eta') \}$, then
\begin{align}
\bb P(B_U \cap F_1^c \cap F_2^c) \le &  \sum_{\gamma \in (0 ,
2+\eta)} \sum_{k \in [\eta m/2r , m]} C_{4.2} \exp(
(\Delta_1+\Delta_2) m )
\nonumber\\
& \quad \cdot \exp\left[-\left\{\left(1+ \frac{\eta+\gamma}{2}\right)k +(1+\eta)(m-k)\right\}\log(n/m) \right].
\end{align}
To simplify the second exponential we drop the $\gamma/2$ from the first term and reduce the
$\eta$ to $\eta/2$ in the second in order to combine them into $(1+\eta/2)m$.
Noting that there are fewer than $rm$ terms in the sum over $\gamma$ and at
 most $m$ terms in the sum over $k$, and using the inequality $m^2
 \le e^m$ for $m \ge 0$, the above is
\begin{align}
 &\le  C_{4.2} r m^2 \exp\left[-(1+\eta/2) m \log(n/m) + (\Delta_1+\Delta_2)
 m\right] \notag \\
 & \le C_{4.2} r \exp\left[-(1+\eta/2) m \log(n/m) + (1+\Delta_1+\Delta_2)
 m\right].\label{B_Ubd1}
\end{align}
 Recalling that $E^{*1}(m , \le (r-1-\eta)m) = \cup_{\{U: |U|= m\}} B_U$ we have
$$
\bb P(E^{*1}(m , \le (r-1-\eta)m)) \le \bb P(F_1)+\bb
P(F_2)+\sum_{\{U: |U|= m\}}
 \bb P(B_U\cap F_1^c\cap F_2^c).
$$
Combining \eqref{F1bound}, \eqref{F2bound} and \eqref{B_Ubd1}, and
using $\binom{n}{m} \le \exp(m\log(n/m)+m)$ from Lemma
\ref{subsetsizebd} we see that  the above is
\begin{align}
\le & \exp[-(\eta/2) m \log(n/m) + (1+\Delta_2) m]
\notag\\
& + C_{4.2}\exp\left[-(\eta^2/4r) m \log(n/ m)+ (2+\Delta_1)
m\right]
\notag\\
& + C_{4.2} r \exp\left[-(\eta/2) m\log(n/m) + (2+\Delta_1+\Delta_2)
m \right]
\notag\\
& \le C \exp[-(\eta^2/4r) m\log(n/m) + (2+\Delta_1+\Delta_2) m],
\label{Hbd1}
\end{align}
which is the desired result.
\end{proof}

\section*{Acknowledgement}

The authors would like to thank the referee for Annals of Probability whose many insightful comments allowed us
to strengthen our results and considerably simplify their proofs. We would like to also think the two referees
at Stochstic Processes and their Applications, who help us to clean up the exposition, and inspired us to
generalize the results to $\theta\ge2$.

\section*{References}

\frenchspacing

\mn
Balogh, J., and Pittel, B.G. (2006) Bootstrap percolation on the random regular graph.
{\it Rand. Struct. Alg.}

\mn
Balogh, J., Peres, Y., and Pete, G. (2006) Bootstrap percolation on infinite trees and non-amenable graphs.
{\it Comb. Prob. Comput.} 15, 715--730

\mn
C.H. Bennett, and G. Grinstein (1985) Role of irreversibility in stabilizing complex and nonergodic behavior in localy interacting discrete systems. {\it Phys. Rev. Letters.} 55, 657--660

\mn 
Chatterjee, S., and Durrett, R. (2009) Contact Processes on
Random Graphs with Power law Degree Distributions Have Critical
Value 0. {\it Ann. Probab.} 37, 2332-2356.

\mn
Chen, H.N. (1992) On the stability of a population growth model with sexual reproduction on $\ZZ^2$. {\it Annals of Probability.} 20, 232--285

\mn
Chen, H.N. (1994) On the stability of a population growth model with sexual reproduction on $\ZZ^2$. {\it Annals of Probability.} 22, 1195--1226

\mn
da Silva, E.F., and de Oliveira, M.J. (2011) Critical discontinuous phase transition in the threshold
contact process. 44, no. 13, paper 135002

\mn
Durrett, R. (1985) Some peculiar properties of a particle system with sexual reproduction.
in {\it Stochastic Spatial Processes.} Edited by P. Tautu. Lecture Notes in Math 1212, Springer, New York  

\mn 
Durrett, R. (2007) {\it Random graph dynamics.} Cambridge U. Press.

\mn
Durrett, R. (2009) Coexistence in Stochastic Spatial Models. (Wald Lecture Paper).
{\it Ann. Appl. Prob.} 19, 477--496

\mn
Durrett, R., and Gray, L. (1985) Some peculiar properties of a particle system with sexual reproduction. Unpublished manuscript.

\mn
Durrett, R. and Jung, P. (2007) Two phase transitions for the contact
process of small worlds. {\it Stoch. Proc. Appl.} 117, 1910--1927

\mn
Durrett, R., and Levin, S. (1994) The importance of being discrete (and spatial).
{\it Theoret. Pop. Biol.} 46 (1994), 363-394

\mn
Durrett, R. and Neuhauser, C. (1994) Particle systems and reaction-diffusion
equations. {\it Ann. Probab.} 22, 289--333.

\mn
Fontes, L.R.G., and Schonmann, R.H. (2008a) Threshold $\theta \ge 2$ contact processes on homogeneous trees.
{\it Probab. Theory Related Fields.} 141, 513--541

\mn
Fontes, L.R.G. and Schonmann, R.H. (2008b) Bootstrap percolation on homogeneous trees has two phase transitions.
{\it J. Stat. Phys.} 132, 839--861

\mn
Grassberger, P. (1981) Monte Carlo simulations of Schl\"ogl's second model.
{\it Physcis Letters A.} 84, 459--461 

\mn
Grassberger, P. (1982) On phase transitions in Schl\"ogl's second model.
{\it Z. Phys. B} 47, 365--374

\mn
Guo, X. Evans, J.W. and Liu, D.J. (2007) Generic two-phase coexistence, relaxation kinetics,
and interface propagation in the quadratic contact process: Analytic sutides.
{\it Physica A.} 387, 177--201

\mn
Guo, X. Evans, J.W. and Liu, D.J. Generic two-phase coexistence, relaxation kinetics,
and interface propagation in the quadratic contact process: Simulation studies.
{\it Phys. Rev. E.} 75, paper 0611129

\mn
Guo, X., Liu, D.J. and Evans, J.W. (2009) Schloegl's second model for autocatlaysis with particle 
diffusion: lattice-gas realization exhibiting generic two-phase coexistence.
{\it J. Chem. Phys.} 130, paper 074106

\mn
Handjani, S. (1997) Survival of threshold contact processes. {\it J. Theoret. Probab.} 10, 737--746

\mn 
Janson, S., {\L}uczak, T. and Ruci{\'n}ski, A. (2000) {\it
Random graphs.} John Wiley and Sons, New York

\mn
Liggett, T.M. (1985) {\it Interacting Particle Systems.} Springer-Verlag, New York

\mn
Liggett, T.M. (1994) Coexistence in threshold voter models. {\it Ann. Probab.} 22, 764--802

\mn
Liggett, T.M. (1999) {\it Stochastic interacting systems: contact, voter, and
exclusion processes.} Springer-Verlag, New York

\mn
Liu, D.J. (2009) Generic two-phase coexistence and nonequilibrium criticality in a lattice
version of Schl\"ogl's second model for autocatlysis. {\it Jouranl Statistical Physics.} 135: 77--85

\mn
Liu, D.J., Guo, X., and  Evans, J.W.  (2007) Quadratic contact process: Phase separation with
interface-orientation-dependent equistability. {\it Phys. Rev. Letters.} 98, paper 050601

\mn
Lubetzky, E. and Sly, A. (2010) Cutoff phenomena for random walks on random regular graphs.
{\it Duke Math. J.} 153), 475–-510

\mn
Mountford, T., and Schonmann, R.H. (2008) The survival of large dimensional threshold contact processes.
{\it Ann. Probab.} 37 (2009), 1483-1501

\mn
Prakash, S. and Nicolis, G. (1997) Dynamics of Schl\"ogl models on lattices of low spatial dimension.
{\it Journal of Statistical Physics.} 86, 1289--1311

\mn
Schl\"ogl, F. (1972) Chemical reaction models for non-equilibrium phase transitions. Z. Physik. 253, 147--161 

\mn
Toom, A.L. (1974) Nonergodic multidimensional systems of automata. {\it Probl. Inform. Transmission} 10, 239--246

\mn
Toom, A.L. (1980) Stable and attractive trajectories in multicomponent systems. Pages 549--575 in
{\it Multicomponent Random Systems.} Dekker, New York

\mn
Watts, D.J., and Strogatz, S.H. (1998) Collective dynamics of `small-world' networks. {\it Nature.} {\bf 393}, 440--442

\end{document}